\documentclass[12pt]{amsart}
\usepackage{amsmath,amssymb,epsfig,color}
\newtheorem{theorem}{Theorem}[section]

\newtheorem{corollary}[theorem]{Corollary}
\newtheorem{definition}[theorem]{Definition}
\newtheorem{lemma}[theorem]{Lemma}
\newtheorem{proposition}[theorem]{Proposition}
\theoremstyle{remark}
\newtheorem{remark}[theorem]{Remark}
\newtheorem{example}[theorem]{Example}

\newcommand{\vanish}[1]{}\parskip=12pt

\begin{document}
\title{A Gray Code for the Shelling Types of the Boundary of a Hypercube} 
\author{Sarah Birdsong, G\'abor Hetyei}
\address{Department of Mathematics and Statistics, UNC Charlotte, 
	Charlotte, NC 28223}
\email{ghetyei@uncc.edu, sjbirdso@uncc.edu}
\date{\today}
\begin{abstract}
We consider two shellings of the boundary of the hypercube equivalent if one can be transformed into the other by an isometry of the cube.  We observe that a class of indecomposable permutations, bijectively equivalent to standard double occurrence words, may be used to encode one representative from each equivalence class of the shellings of the boundary of the hypercube. These permutations thus encode the shelling types of the boundary of the hypercube. We construct an adjacent transposition Gray code for this class of permutations. Our result is a signed variant of King's result showing that there is a transposition Gray code for indecomposable permutations. 
\end{abstract}
\maketitle


\section*{Introduction}

There is a significant amount of research devoted to finding Gray
codes for classes of permutations, where two permutations are considered
adjacent if they differ by an involution of some special kind. For a
survey of some key results we refer the reader to Savage's
paper~\cite[Section 11]{Savage}. The simplest and most elegant result
in this area is the Johnson-Trotter algorithm~\cite{Johnson, Trotter},
providing an adjacent transposition Gray code for all permutations of a
finite set. 

The present work is motivated by King's recent paper~\cite{King}, providing a
transposition Gray code for the set of all indecomposable permutations
of a finite set. The combinatorial interest in these permutations is
long-standing, both Comtet~\cite[p.\ 261]{Comtet-AC} and
Stanley~\cite[Ch.\ 1, Exercise 32]{Stanley-EC1} discuss them in their 
textbooks. As shown by Ossona de Mendez and
Rosenstiehl~\cite{OM-Rosenstiehl-conn}, they are bijectively equivalent to
rooted hypermaps. Refining Dixon's famous result~\cite{Dixon}, stating
that a random pair of permutations generates almost always a transitive
group is also related to enumerating indecomposable permutations, see
Cori's work~\cite{Cori-indec}.  

Our main result is finding an adjacent transposition Gray code for a
signed variant of the indecomposable permutations, which we call {\em
  sign-connected permutations of the set $\{\pm 1,\pm 2,\ldots,
  \pm n\}$}. The importance of this signed variant is highlighted by the
fact that our sign-connected permutations encode the shelling types of
the boundary of the hypercube. Here we consider two shellings to be of
the same type, if they may be transformed into each other by an isometry
of the underlying hypercube. Adjacent transpositions correspond to
swapping adjacent entries in a list of facets.

We achieve our goal in two steps. First, in Section~\ref{sec:gray}, we provide a
very simple adjacent transposition Gray code for a larger class of 
permutations, which we call {\em standard permutations of the set
  $\{\pm 1,\pm 2,\ldots, \pm n\}$}. The simplicity of this code is
comparable to the simplicity of the output of the Johnson-Trotter
algorithm. The second step, contained in Section~\ref{sec:conn},  is to
show the following, surprisingly simple statement: an adjacent
transposition Gray code for the class of sign-connected permutations of
$\{\pm 1,\pm 2,\ldots, \pm n\}$ may be obtained from the Gray code of
all standard permutations by simply deleting those standard permutations  
of $\{\pm 1,\pm 2,\ldots, \pm n\}$ which are not
  sign-connected. Although the statement is very simple, its proof
  requires a careful look at a special way of coding standard
  permutations of $\{\pm 1,\pm 2,\ldots, \pm n\}$. This coding is
    inspired by the place-based inversion tables, introduced for
    ordinary permutations in~\cite{Hetyei}. Our result justifies the
    hope that the 
    same coding applied to ordinary permutations may help us find an
    adjacent transposition Gray code for indecomposable permutations, 
if such a Gray code exists.  

Our standard permutations are bijectively equivalent to
{\em standard double-occurrence words} with $n$ symbols. As shown by
Ossona de Mendez and Rosenstiehl~\cite{OM-Rosenstiehl}, these words are
bijectively equivalent to rooted maps with $n-1$ edges, this result was
later refined by Drake~\cite{Drake}. Standard double-occurrence words are
also bijectively equivalent to fixed point free involutions, as defined by
Cori~\cite{Cori}. Under these bijections, sign-connected standard
permutations correspond to connected rooted maps and indecomposable
fixed-point free involutions. Our Gray codes may be useful in the study
of these related objects.  

Our paper is structured as follows. In the Preliminaries we gather a few 
basic facts on permutation Gray codes and shellings of the boundary of a
hypercube that we will need. We also define sign-connected permutations
of $\{\pm 1,\pm 2,\ldots, \pm n\}$. In Section~\ref{sec:hyp} we
introduce standard permutations of $\{\pm 1,\pm 2,\ldots, \pm n\}$ and show
that the sign-connected standard permutations are bijectively equivalent
to the shelling types of the boundary of an $n$-dimensional
hypercube. The fact that standard sign-connected permutations correspond
to standard double-occurrence words is shown in Section~\ref{sec:arc}
where we also introduce arc diagrams that help visualize standard
permutations of $\{\pm 1,\pm 2,\ldots, \pm n\}$. Essentially the same
visualization was used by 
Drake~\cite{Drake}. Inspired by the place-based noninversion tables 
introduced in~\cite{Hetyei}, we introduce a new encoding of our arc
diagrams. It is this encoding that makes the definition of the simple
Gray code for all standard permutations in Section~\ref{sec:gray}
truly easy. Essential properties of this Gray
code are explored in Section~\ref{sec:prop}. Finally, our main result
may be found in Section~\ref{sec:conn}.

Our result highlights a connection between the study of
hypermaps and the theory of shellings which may be worth exploring
further in the future. It also raises the hope that, by using a similar
encoding to the one introduced in~\cite{Hetyei}, one could find an
adjacent transposition Gray code for indecomposable
permutations. Finally, although a large amount of literature exists on shelling
and shellability, very little has been done to explore the set of all
shellings of the same object. Our paper is among the first in this
new direction.

\section{Preliminaries}

\subsection{Gray codes for sets of permutations}
Given a finite set $X$ with $n$ elements, we define a {\em permutation} of $X$ as a word $\pi=\pi(1)\cdots \pi(n)$ in which every element of $X$ appears exactly once. We denote the set of all permutations of $X$ by ${\mathcal
  S}_X$. Given any subset ${\mathcal V}$ of ${\mathcal S}_X$, we consider the following graph on the vertex set ${\mathcal V}$: the unordered pair $\{\pi,\rho\}\subseteq {\mathcal V}$ is an edge if $\pi$ and $\rho$
differ by an {\em adjacent transposition}, i.e., there is an $i\in \{1,\ldots,n-1\}$ such that $\rho(i+1)=\pi(i)$, $\rho(i)=\pi(i+1)$, and $\rho(j)=\pi(j)$ for $j\in\{1,\ldots,n\}\setminus \{i,i+1\}$. An {\em adjacent transposition Gray code for ${\mathcal V}$} is a Hamiltonian path in the resulting graph. Obviously not all sets of permutations have
such a Gray code.  

The Johnson-Trotter algorithm generates an adjacent transposition Gray
code on the set ${\mathcal S}_X$ of all permutations of the set $X$.  When $X=\{1,2,\dots,n\}$, this algorithm recursively defines the Gray code as follows: set the initial permutation as $12...n$; replace each permutation in the Gray code for permutations of length $n-1$ with $n$ new permutations of length $n$; for each of the length $n-1$ permutations, insert $n$ into every position $1$ through $n$ to get the new permutations of length $n$.  For odd-indexed, length $n-1$ permutations, insert $n$ from right to left, and left to right for even-indexed permutations~\cite{Johnson, Trotter}.  See Table~\ref{tab:JT} for the Gray code when $n=2$ and $n=3$. 

\begin{table}[h]
\caption{The Gray code produced by the Johnson-Trotter algorithm for $n=2$ and $n=3$ (read down)}
\label{tab:JT}
\begin{tabular}{cccc}
\hline
$n=2$ &      &  $n=3$ &     \\
\hline
12  &       & 123 & 321  \\
21  &       & 132 & 231  \\
      &       & 312 & 213  \\
\hline
\end{tabular}
\end{table}

A permutation $\pi$ of ${\mathcal S}_n$ is {\em connected} if there is no $m < n$ such that $\pi$ sends ${\{1,2,\ldots ,m\}}$ into ${\{1,2,\ldots ,m\}}$.  King found a {\em transposition Gray code} for such connected permutations~\cite{King}, i.e., a Gray code for the graph whose vertices are connected permutations and whose edges connect permutations that differ by a (not necessarily adjacent) transposition. It is still open, whether there is an adjacent transposition Gray code for connected permutations. 

In this paper, we be interested in the following signed variant of connected permutations.

\begin{definition}
\label{def:sign-conn}
We call a permutation $\pi=\pi(1) \cdots \pi(2n)$ of $\{\pm 1, \pm 2, \ldots , \pm n\}$ {\em sign-connected} if and only if for all $1\leq m<2n$ there is at least one $j\in \{1,2,\ldots,n\}$ such that $|\{\pi(1),\ldots,\pi(m)\}\cap
\{-j,j\}|=1$. If a permutation of $\{\pm 1, \pm 2, \ldots , \pm n\}$ is not sign-connected then we call it {\em sign-disconnected}.
\end{definition}

As we will see in Section~\ref{sec:hyp}, a subset of these permutations is identifiable with the shelling types of the boundary of a hypercube. We will show in Section~\ref{sec:conn} that there is an adjacent transposition Gray code for this subset.    

\subsection{Shellings of the boundary of the hypercube}
\label{sec:shell}

We define the standard $n$-dimensional hypercube ($n$-cube) to be $[-1,1]^n \subset \mathbb{R}^n$.  As observed by Metropolis and Rota~\cite{Metropolis-Rota}, each non-empty face of the standard hypercube may be denoted by a vector $(u_1,\ldots,u_n)\in \{-1,*,1\}^n$, where $u_i=-1$ or $u_i=1$ indicates that all points in the face have the $i^{th}$ coordinate equal to $u_i$; whereas, $u_i=*$ indicates the $i^{th}$ coordinate of the points in the face range over the entire set $[-1,1]$. Using this notation, the code for each facet of the boundary is of the form $(u_1,\ldots,u_n)\in \{-1,*,1\}^n$ such that exactly one $u_i$ belongs to $\{-1,1\}$. Thus, we can encode each facet by a single number $-k$ or $k$ where $k\in\{1,\ldots,n\}$ is the unique index $k$ such that $u_k=\pm 1$.  We write $-k$ or $k$ when $u_k=-1$ or $u_k=1$, respectively. Using this notation, we will identify each enumeration of the facets of the boundary of the $n$-cube with a permutation of the set $\{\pm 1, \pm2, \ldots, \pm n\}$. This correspondence is a bijection.  

A {\em shelling} is a particular way of listing the facets of the boundary of a polytope, see \cite[Definition 8.1]{Ziegler} for a definition.  The fact that the boundary complex of any polytope is shellable was shown by Bruggeser and Mani~\cite{Bruggeser-Mani}. A polytope is {\em cubical} if each of its proper faces is combinatorially equivalent to a cube. The boundary complex of a cubical polytope is a cubical complex; see Chan's paper~\cite{Chan} for a definition of a cubical complex as the geometric realization of a cubical poset. The boundary complex of a cubical polytope is {\em pure}, i.e., all facets have the same dimension. Following Chan~\cite{Chan}, we define a shelling of a $(d-1)$-dimensional cubical complex as an enumeration $(F_1,\ldots,F_r)$ of its facets in such a way that for all $m>1$ the intersection $F_m\cap (F_1\cup\cdots\cup F_{m-1})$ is a union of $(d-2)$-faces homeomorphic to the ball or sphere. Such an intersection is a {\em cubical shelling component}, whose {\em type} is $(i,j)$ if it is the union of $i$ antipodally unpaired $(d-2)$-faces and $j$ pairs of antipodal $(d-2)$-faces. The ``only if part'' of the following statement is obvious and used in Chan's work~\cite{Chan} without proof.  Both implications are shown by Ehrenborg and Hetyei~\cite[Lemma 3.3]{Ehrenborg-Hetyei}. 

\begin{lemma}[Ehrenborg-Hetyei]
\label{l:EH}
The ordered pair $(i,j)$ is the type of a shelling component in a shelling of a cubical $(d-1)$-complex if and only if one of the following holds:
\begin{itemize}
\item[(i)] $i=0$ and $j=d-1$; or 
\item[(ii)] $0<i<d-1$ and $0\leq j\leq d-1-i$.
\end{itemize}
Furthermore, in case (i), the shelling component is homeomorphic to a $(d-2)$-sphere, and in case (ii), the shelling component is homeomorphic to a $(d-2)$-ball. 
\end{lemma}

As a consequence of Lemma~\ref{l:EH}, we may describe the shellings of the boundary complex of an $n$-cube in the following way.

\begin{lemma}
An enumeration $(F_1,\ldots,F_{2n})$ of the facets of the boundary complex of the $n$-cube is a shelling if and only if for each $m<2n$, the set $\{F_1,\ldots, F_m\}$ contains at least one antipodally unpaired facet.  
\end{lemma}
\begin{proof}
On the one hand, the cubical complex $F_1\cup\cdots\cup F_m$ is shellable and $(n-1)$-dimensional, and as such it is homeomorphic to a $(n-1)$-ball or an $(n-1)$-sphere. Since the boundary complex $F_1\cup \cdots F_{2n}$ is an $(n-1)$-sphere, the proper subcomplex $F_1\cup\cdots\cup F_m$ can only be an $(n-1)$-ball. By \cite[Lemma 3.3]{Ehrenborg-Hetyei} (see part (i) of Lemma~\ref{l:EH} above), $F_1\cup\cdots\cup F_m$ must contain at least one antipodally unpaired facet. 

Conversely, assume that $F_1\cup\cdots\cup F_m$ contains at least one antipodally unpaired facet for each $m<2n$. In other words, each $F_1\cup\cdots\cup F_m$ could be used as a shelling component in the shelling of an $n$-dimensional cubical complex, and its type $(i_m,j_m)$ satisfies $i_m>0$. Adding $F_m$ to $F_1\cup\cdots\cup F_{m-1}$ results either in introducing a new antipodally unpaired facet or $F_m$ is the antipodal pair of a previously listed facet. In the first case, we have $i_m=i_{m-1}+1$ and $j_m=j_{m-1}$, so $F_m\cap (F_1\cup\cdots\cup F_{m-1})$ is a shelling component of type $(i_m,j_m)$. In the second case, we have $i_m=i_{m-1}-1$ and $j_m=j_{m-1}+1$, so $F_m\cap (F_1\cup\cdots\cup F_{m-1})$ is a shelling component of type $(i_m,j_m)$. Finally, for $m=2n$, $F_{2n}\cap (F_1\cup\cdots\cup F_{2n-1})$ is a shelling component of type $(0,n-1)$. 
\end{proof}

For each initial segment of an enumeration of the facets of a shelling
of the boundary of the $n$-cube, there is at least one antipodal pair of
facets such that exactly one of the two facets belongs to the shelling
component, i.e., the facet $k$ is listed in the first $i$ facets of the
shelling but facet $-k$ facet is not~\cite{Ehrenborg-Hetyei}. 

\begin{corollary}
\label{cor:conn}
A permutation of $\{\pm 1, \pm 2, \ldots , \pm n\}$ is sign-connected if and only if it represents a shelling of the facets of the hypercube.  Similarly, a sign-disconnected permutation represents an enumeration of the facets which is not a shelling.
\end{corollary}

As a result of Corollary~\ref{cor:conn}, the number of sign-connected permutations equals the number of shellings of the boundary of the $n$-cube.  Given any $n$, this number can be found by the recursive formula $a_n = (2n-1)!!-\sum_{k=0}^{n-1} (2k-1)!! \cdot a_{n-k}$.  The sequence $\{a_n\}$ is sequence A000698 in the On-Line Encyclopedia of Integer Sequences~\cite{A000698}.  

\section{Equivalence Classes of the Enumerations of the Facets of the $n$-Cube}
\label{sec:hyp}

The isometries of the $n$-cube permute its facets, inducing a $B_n$-action on the enumerations of all facets of the boundary of the $n$-cube.  This action is free, i.e., any nontrivial isometry takes each enumeration into a different enumeration.

\begin{definition}
We consider two enumerations of the facets of the $n$-cube equivalent if they can be transformed into each other by an isometry of the $n$-cube.  
\end{definition}

As noted in Section~\ref{sec:shell}, every enumeration of the facets of
the boundary of the $n$-cube can be identified with a signed
permutation.  The induced action of $B_n$ on these signed permutations
is generated by the following operations: 
\begin{enumerate}
\item For each $k\in\{1,\ldots,n\}$ there is a reflection
  $\varepsilon_k$  interchanging $k$ with $-k$ and leaving all other
  entries unchanged;  
\item For each $\{i,j\}\subset\{1,\ldots,n\}$ there is a
  reflection $\rho_{i,j}$ interchanging $i$ with $j$ and $-i$ with $-j$,
  leaving all other entries unchanged.   
\end{enumerate} 
It is worth noting that we may also identify the elements of $B_n$ with
signed permutations the usual way, the action of $B_n$ we consider is
then the action of $B_n$ on itself, via conjugation. There exist $2^n
\cdot n!$ symmetries of the $n$-cube~\cite{Humphreys}.  Since the $B_n$
action is free, all equivalence classes have the same cardinality,
giving a total of $\dfrac{(2n)!}{2^n \cdot n!}=(2n-1)!!$ equivalence
classes. 

\begin{definition}
\label{def:std}
A standard permutation is defined to be an permutation of $\{\pm1,\ldots ,\pm n\}$ such that the following two properties hold:
\begin{enumerate}
\item for all  $i$, $i$ occurs before $-i$ in the list, and 
\item the negative numbers in the list appear in the following order:
  $-1, -2, \ldots , -n$. 
\end{enumerate}
\end{definition}

Since there are $2^n$ ways to assign the negative sign to the pair $k$ and $-k$ and $n!$ ways to order $\{-1,\ldots ,-n\}$, we have a total of $\dfrac{(2n)!}{2^n \cdot n!}=(2n-1)!!$ standard permutations. 

\begin{lemma}
Each equivalence class of the enumerations of the facets of an $n$-cube corresponds to exactly one standard permutation.
\end{lemma}
\begin{proof}
Since the number of equivalence classes and the number of standard permutations are both $(2n-1)!!$, we need only to show that every $\pi \in {\mathcal S}_{\{\pm 1,\ldots ,\pm n\}}$ is equivalent to a standard permutation. 

Pick any $\pi \in {\mathcal S}_{\{\pm 1,\ldots ,\pm n\}}$, and suppose
$\Pi$ is the equivalence class which contains $\pi$.  
Applying only reflections $\varepsilon_k\in B_n$ we may replace $\pi$
with a $\pi'\in \Pi$ which satisfies condition (1) in
Definition~\ref{def:std}. Applying only reflections $\rho_{i,j}\in
B_n$ we may replace $\pi'$ with a $\pi''\in \Pi$ that also satisfies  
condition (2) in Definition~\ref{def:std}. Note that the application of
an operator $\rho_{i,j}$ leaves the validity of condition (1) unchanged.
\end{proof}

The set of sign-connected permutations is closed under the action of
$B_n$; thus, we can think of equivalence classes of shellings as types
of shellings. The following characterization of standard sign-connected
permutations is an immediate consequence of
Definitions~\ref{def:sign-conn} and \ref{def:std}.
\begin{lemma}
\label{l:std-conn}
A standard permutation $\pi=\pi(1) \ldots  \pi(2n)$ of $\{\pm 1, \pm 2, \ldots , \pm n\}$ is sign-disconnected if and only if there exists an $i$ such that $|\pi(j)| \le i$ for all $j \le 2i$ (and $|\pi(j)| \ge i+1$ for all $j > 2i$).  
\end{lemma}

Applying the definition of a standard permutation to our previous notion of a sign-connected permutation, we get the following characterization.

\begin{lemma}
\label{l:ineq}
A standard permutation $\pi=\pi(1) \cdots \pi(2n)$ is sign-connected if
and only if for all $m<2n$, the sum $\pi(1)+\cdots+\pi(m)> 0$.
\end{lemma}
\begin{proof}
By definition of a standard permutation, each $j$ appears before $-j$,
thus we have $\pi(1)+\cdots+\pi(m)\geq 0$ for every $m\leq 2n$. Equality
occurs exactly when the set $\{\pi(1),\ldots,\pi(m)\}$ is the union of
pairs of the form $\{j,-j\}$. The existence of an $m<2n$ satisfying 
$\pi(1)+\cdots+\pi(m)=0$ is thus equivalent to the standard permutation
being sign-disconnected. 
\end{proof}

\begin{remark}
\label{rem:ineq}
Note that the proof of the ``if'' part of Lemma~\ref{l:ineq} may be
restated in the following, stronger from: $\pi=\pi(1) \cdots \pi(2n)$ is
sign-connected if for all $m<2n$ we have $\pi(1)+\cdots+\pi(m)\neq 0$. 
\end{remark}

The standard permutation in Example~\ref{ex:discon} is sign-disconnected.

\section{Arc Diagrams}
\label{sec:arc}

A standard permutation $\pi \in {\mathcal S}_{\{\pm 1,\ldots, \pm n\}}$
of the facets of the boundary of the $n$-cube can be visually represented by an arc
diagram.  The arc diagram is constructed as follows: put $2n$ vertices
in a row; label them left to right with $\pi(1),\ldots, \pi(2n)$; 
then for each $k\in\{1,\ldots,n\}$, create an arc connecting the
vertices labeled $-k$ and $k$.  See Fig.~\ref{fig:arc01} for an example
when $n=3$. Hence, in the associated arc diagram of a standard permutation,
the order of the labels of the vertices will match the order of the standard
permutation.

Note that the labels of the vertices in the arc diagram are uniquely
determined by the underlying complete matching.  This matching is 
represented by the arcs where the right endpoints of the arcs must be 
labeled left to right by $-1,-2,\ldots, -n$, in this order, and the left 
endpoint of the arc whose right end is labeled $-k$ must be be labeled 
$k$. Hence, using these rules, we can uniquely reconstruct the
associated standard permutation from its the arc diagram. Thus, the arc diagram
representation provides a bijection  
between standard permutations in ${\mathcal S}_{\{\pm 1,\ldots, \pm n\}}$ 
and complete matchings of a $2n$ element set. 

\begin{figure}[h]
\input{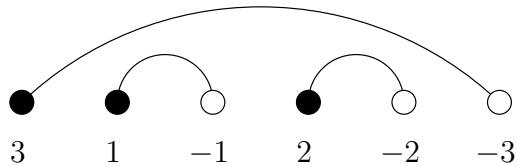}
\caption{The arc diagram associated to $(3,1,-1,2,-2,-3)$}
\label{fig:arc01}
\end{figure}

Almost the same bijection appears in the work of
Drake~\cite{Drake}, the only difference being that Drake encodes
complete matchings with {\em standard double occurrence words} in the letters
$1,2,\ldots,n$. A double occurrence word is a word  
in which each letter occurs exactly twice. Ossona de Mendez and
Rosenstiehl~\cite{OM-Rosenstiehl} call a double occurrence word
{\em standard} if the first occurrence of the letters happens in
increasing order. (Note that Drake~\cite{Drake} omits the adjective
standard, but he adopts his terminology from~\cite{OM-Rosenstiehl}, and
the words he uses to encode complete matchings are the standard
double occurrence words.) 

Taking the {\em reverse complement} of such a double occurrence word
and changing the second occurrence of each $k$ to $-k$ results in a
standard permutation. This correspondence is a bijection. (As usual,
the complement of a word in the letters $1,2,\ldots,n$ is the word
obtained by replacing each letter $i$ with $n+1-i$, and the reverse of a
word $w_1 w_2\cdots w_{2n}$ is the word $w_{2n}w_{2n-1}\cdots w_1$.) For
example, the standard permutation $(3,1,-1,2,-2,-3)$ corresponds to the
same matching as the double occurrence word $122331$ in Drake's
work~\cite{Drake}. 

\begin{remark}
Let us associate to each standard permutation $\pi$ of $\{\pm 1, \pm 2,
\ldots, \pm n\}$  
the {\em fixed point free involution $\widehat{\pi}$} of $\{\pm 1, \pm
2, \ldots , \pm n\}$  
that exchanges the elements $\pi^{-1}(-i)$ and $\pi^{-1}(i)$ for
$i=1,2,\ldots,n$.  
It is easy to see that this correspondence is a bijection. Under this
bijection, sign-connected standard permutations correspond to
indecomposable fixed point free involutions, as defined by Cori~\cite{Cori}. 
\end{remark}

For each $k\in\{1,\ldots,n\}$, the definition of a standard permutation
forces the vertex labeled $k$ to be to the left of the vertex labeled
$-k$ in the arc diagram.  Any arc diagram associated to a standard
permutation can be represented by a word $a_1 a_2 \ldots  a_n$,
recursively constructed as follows. Let $a_n$ be the position of the
rightmost arc's left endpoint, i.e., $a_n$ is the position of the
vertex labeled $n$.  Remove the rightmost arc from the diagram, and repeat
the process until all arcs are removed.  As a result of how we defined
this representation, $1 \le a_i \le 2i-1$ for each $a_i$ in the word
$a_1 a_2 \ldots  a_n$, and each word of this form corresponds to exactly
one standard permutation.  Conversely, every such word $a_1\ldots a_n$
will encode some arc diagram. 
\begin{example}
\label{ex:con}
$(3,1,-1,2,-2,-3) \cong 131$.
\end{example}

\begin{definition}
We will denote the standard permutation which the word $a_1\cdots a_n$ encodes by $\pi(a_1\cdots a_n)$.
\end{definition}

\begin{lemma}
\label{lem:discon}
A standard permutation $\pi(a_1 a_2 \ldots  a_n)$ is sign-disconnected if and only if there exists a $k \ge 2$ such that $a_k=2k-1$ and $a_j \ge a_k = 2k-1$ for all $j > k$.
\end{lemma}
\begin{proof}
Suppose $a_1\ldots a_n$ encodes a sign-disconnected standard permutation $\pi(1)\ldots \pi(2n)$.  Then by Lemma~\ref{l:std-conn}, there exists an $i$ such that $|\pi(j)| \ge i+1$ for all $j > 2i$, meaning that in the associated arc diagram both ends of the $(i+1)^{st}$ through $n^{th}$ arcs are located at vertex positions $2i+1$ or greater.  Hence, $a_{i+1}=2i+1$ and $a_j \ge 2i+1$ for $j > i$.

Conversely, let $a_1 \ldots  a_n$ encode a standard permutation $\pi(1) \ldots  \pi(2n)$.
Suppose $k$ is the smallest integer such that $k\geq 2$, $a_k=2k-1$ and
$a_j \ge a_k = 2k-1$ for all $j > k$. 
Then in the associated arc diagram, the right endpoint of the $(k-1)^{st}$ arc is located at vertex position $2k-2$.  Thus, $\pi(2k-2)=-(k-1)$ since the right endpoint of any arc is labeled with the negative number of the antipodal pair.
Because $\pi(1) \ldots  \pi(2n)$ is a standard permutation, we have $|\pi(j)| \le k-1$ for all $j \le 2k-2$.
Also, $a_j \ge 2k-1$ for all $j \ge k$ in the word encoding $\pi$'s associated arc diagram means that the left endpoints of the $j^{th}$ arcs are all at vertex position $2k-1$ or greater.
Hence, $|\pi(j)| \ge k$ for all $j \ge 2k-1$.
\end{proof}

Clearly, a standard permutation is sign-disconnected if and only if
a vertical line can be drawn between the first and last vertices of the 
associated arc diagram which does not intersect any of the arcs 
and which separates two adjacent vertices. In other words, the
arc diagram is comprised of more than one component of overlapping arcs.
Drake~\cite{Drake} noted that this occurs if there exists a vertex $k$,
$1<k<2n$, which is not nested under any arc in the diagram.  
We will refer to the arc diagram of a sign-connected standard permutation as
connected and to the arc diagram of a sign-disconnected standard
permutation as disconnected. An arc diagram is connected  
exactly when it represents a connected matching as defined by
Drake~\cite{Drake}. In terms of standard double occurrence words, connected arc
diagrams correspond to connected standard double occurrence words in the
work of Ossona de Mendez and Rosenstiehl~\cite{OM-Rosenstiehl}.

We will call a word $a_1 \cdots a_n$ {\it arc-connected} if and only if
the arc diagram associated to $\pi(a_1 \cdots a_n)$ is connected,
otherwise we call the word {\it arc-disconnected}.  

\begin{example}
\label{ex:discon}
The arc diagram associated to $(1,2,-1,-2,3,-3) \cong 125$, shown in
Fig.~\ref{fig:arc02}, is disconnected. 

\begin{figure}[h]
\input{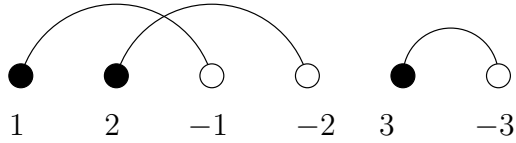}
\caption{The arc diagram associated to $(1,2,-1,-2,3,-3)$}
\label{fig:arc02}
\end{figure}
\end{example}

\begin{corollary}
\label{cor:arc-conn}
A word is arc-disconnected if and only if there exists a $k$ such that $a_k=2k-1$ and $a_j \ge a_k = 2k-1$ for all $j > k$.  Similarly, a word is arc-connected if and only if no such $k$ exists.
\end{corollary}

\begin{definition}
A minimal arc will be defined to be a connected component of an arc diagram such that the component consists of exactly one arc.
\end{definition}

For example, in the arc diagram of Example~\ref{ex:discon} (see Fig.~\ref{fig:arc02}), there is one minimal arc located at the rightmost end of the diagram.  However, in the arc diagram of Example~\ref{ex:con} (see Fig.~\ref{fig:arc01}), there is no minimal arc since the third arc stretches over the first two arcs in the diagram, i.e., $a_3=1$.

The effect of an adjacent transposition of the facets of the boundary of the $n$-cube, i.e., in the standard permutation,  on the associated arc diagram is exactly one swap of adjacent ends of two distinct arcs.  Namely, an interchange of two left ends, or an interchange of one left end and one right end.

If $\pi(a_1\ldots a_n)$ and $\pi(b_1\ldots b_n)$ differ by an adjacent
transposition, then there exists a unique $i$ such that $b_i=a_i \pm 1$
and $b_j=a_j$ for all $j \neq i$.  The converse is not true.  For
example, the word $122$ is obtained from the word $112$ by changing the
second letter by $1$, but the standard permutations $\pi(112)$ and
$\pi(122)$ differ by more than just a single adjacent transposition, see
Figure~\ref{fig:arct01}. 

\begin{figure}[h]
\input{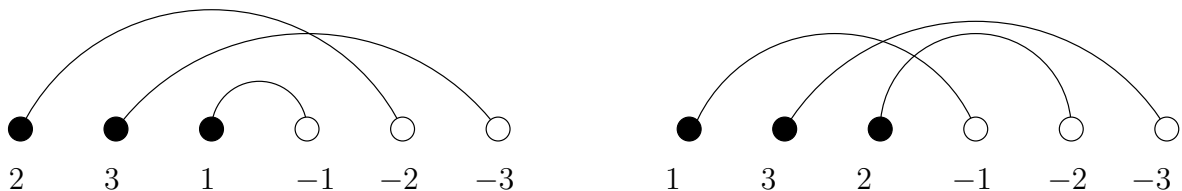}
\caption{The arc diagrams associated to $\pi(112)$ and $\pi(122)$}
\label{fig:arct01}
\end{figure}
  
\section{A Gray Code for all Standard Permutations}
\label{sec:gray}

In this section, we define a Gray code for the standard
permutations of the set $\{\pm 1,\ldots,\pm n\}$ for each $n\geq 1$. We
will write our permutations in the form $\pi(a_1\ldots a_n)$, using the encoding
introduced in Section~\ref{sec:arc}. To simplify our notation, we omit
the operator $\pi$ and write $a_1\cdots a_n$ instead of $\pi(a_1\ldots
a_n)$ in our lists.  We will refer to this simplified way of writing
our Gray code as {\em a Gray code for words}. 

We define our Gray code for words recursively as follows. For $n=1$, we
have only one word to list, namely $1$. To write the Gray code for words
of length $n$, start with $11\ldots 1$. 
Increase $a_n$ by $1$ to get the next word: $11\ldots 12$.  Continue to
increase $a_n$ by $1$ to get a list of codes.  Stop increasing $a_n$
once $a_n=2n-1$ (recall $1\le a_i \le 2i-1$ for any $i$). Now replace
$a_1 \ldots  a_{n-1}$ with the next word in the Gray code for words of
length $n-1$ to get $11\ldots 12(2n-1)$.  Decrease $a_n$ by $1$ to
get the next string of codes. When $a_n=1$, replace $a_1 \ldots
a_{n-1}$ with the next word in the Gray code of length $n-1$.  Continue
in this fashion until the Gray code terminates with $135\ldots (2n-1)$.
See Table~\ref{tab:GC} for the Gray code when $n=2$ and $n=3$. 

\begin{table}[h]
\caption{The Gray Code (read down)}
\label{tab:GC}
\begin{tabular}{ccccc}
\hline
$n=2$ &      &  $n=3$ &   &  \\
\hline
11  &       & 111 & 125 & 131 \\
12  &       & 112 & 124 & 132 \\
13  &       & 113 & 123 & 133 \\
      &       & 114 & 122 & 134 \\
      &       & 115 & 121 & 135 \\
\hline
\end{tabular}
\end{table}

\begin{theorem}
\label{thm:fullgc}
The enumeration defined above is an adjacent transposition Gray code for
the standard permutations of the set $\{\pm1,\ldots,\pm n\}$.
\end{theorem}
\begin{proof}
We proceed by induction on $n$. Comparing two consecutive words
$a_1\ldots a_n$ and $b_1\ldots b_n$, they
differ in exactly one letter. We distinguish two cases, depending on
whether this letter is the last letter of the word or some other letter.

{\noindent\it Case 1:} $a_n\neq b_n$. In this case, the associated standard
permutations differ by an adjacent transposition since the first $n-1$
arcs in the two associated arc diagrams remain stationary, and the
$n^{th}$ arc moves to the right or to the left by one vertex position.
This move switches the left end of the $n^{th}$ arc with an adjacent end
of a different arc. As mentioned at the end of Section~\ref{sec:arc},
switching adjacent ends of two arcs corresponds to an adjacent
transposition on the associated  signed permutation. 

{\noindent\it Case 2:} $a_n=b_n$. This case occurs only if $a_n=1$,
corresponding 
to the $n^{th}$ arc stretching over the first $n-1$ arcs; or if
$a_n=2n-1$, which means the $n^{th}$ arc is a minimal arc.  Thus, in
either situation, the $n^{th}$ arc will not affect whether or not the
move produces an adjacent transposition in the associated standard
permutations.  By the recursive definition of the Gray code, we replace
$a_1 \ldots  a_{n-1}$ with the next word in the Gray code for words of
length $n-1$. By our induction hypothesis, consecutive words in the Gray
code for words of length $n-1$ correspond to standard permutations on
$\{\pm 1,\ldots ,\pm(n-1)\}$ which differ by an adjacent transposition.   
\end{proof}

\section{Properties of the full Gray Code}
\label{sec:prop}

The Gray code for the equivalence classes of the enumerations of the
facets of the boundary of the $n$-cube, referred to by their standard
permutations, will be referred to as the {\em full Gray code}.  As seen
in Section~\ref{sec:gray}, this Gray code may be defined in terms of a
list of standard permutations in the form $\pi(a_1 \ldots a_n)$ for each
word $\underline{a}=a_1\ldots a_n$.  We introduce the notation
$\tau(\underline{a})=a_1\ldots a_{n-1}$ to represent the truncated word,
obtained by removing the last letter of $\underline{a}$. 

\begin{definition}
The collection of all words $\underline{a}$ such that $\tau (\underline{a})=a_1\cdots a_{n-1}$ is fixed will be referred to as a run in the Gray code for words.
\end{definition}

All words in a run form a sublist of codes corresponding to standard permutations differing by an adjacent transposition.  In the run, only $a_n$ changes, either increasing from $1$ to $2n-1$ or decreasing from $2n-1$ to $1$.
A run is increasing if $a_n$ increases to get each subsequent word in the run.  If an increasing run is the $k^{th}$ run in the Gray code, then $k \equiv 1$ mod 2.  Hence, increasing runs will also be referred to as odd runs.  
A run is decreasing if $a_n$ decreases to get each subsequent code in the run.  If a decreasing run is the $k^{th}$ run in the Gray code, then $k \equiv 0$ mod 2.  Decreasing runs will also be referred to as even runs.
Note, odd and even refer to the count of the run and not whether $a_{n-1}$ is odd or even.  For example, $1111$ through $1117$ is the first run in the Gray code for words of length $4$, and $a_3=1$.  However, $12561$ through $12569$ is the $37^{th}$ run in the Gray code for words of length $5$, and $a_4=6$. 

Several properties of the Gray code follow directly from its definition:
\begin{enumerate}
\item  Suppose $\pi(\underline{b})=\pi(b_1\ldots b_n)$ immediately
  follows $\pi(\underline{a})=\pi(a_1\ldots a_n)$ in the Gray code and
  $\underline{a}$ and $\underline{b}$ are in different runs. If
  $\underline{a}$ is in an odd (increasing) run, then $\underline{b}$ is
  in an even (decreasing) run and $a_n=b_n=2n-1$.  If $\underline{a}$ is
  in an even (decreasing) run, then $\underline{b}$ is in an odd
  (increasing) run and $a_n=b_n=1$. 
\item  In every run, there is at least one arc-disconnected word (when $a_n=2n-1$) and at least two arc-connected words (when $a_n=1$ or $2$).
\item  There are $(2n-3)!!$ runs in the Gray code where each word that encodes a standard permutation has length $n$.
\end{enumerate}


\begin{lemma}
\label{fact-m}
Suppose $\pi(\underline{a})=\pi(a_1 \ldots a_n)$ is the $m^{th}$ standard permutation in the Gray code.  
Then $(n-1)+\sum_{i=1}^{n} a_i \equiv m$ mod 2. 
\end{lemma}
\begin{proof}
We will use induction on $m$. For $m=1$, $a_1=a_2=\cdots=a_{n}=1$ and 
$(n-1)+\sum_{i=1}^{n} a_i =2n-1 \equiv 1 =m$ mod 2. Assume the statement is true for the $m^{th}$ standard permutation $\pi(\underline{a})$, and let $\underline{b}$ be the word which encodes the $(m+1)^{st}$ standard permutation in the Gray code. As noted above, there exists a unique $i$ such that $b_i=a_i\pm 1$ and $b_j=a_j$ for all $j\neq i$. Hence the parity of $(n-1)+\sum_{i=1}^{n} b_i$ is the opposite of the parity of $(n-1)+\sum_{i=1}^{n} a_i$, that is $(n-1)+\sum_{i=1}^{n} b_i \equiv m+1$ mod 2. 
\end{proof}

\begin{corollary}
\label{run-parity}
If $\underline{a}$ is in an increasing run, then $a_1+a_2+\ldots +a_{n-1}+(n-2)\equiv 1$ mod 2, and if $\underline{a}$ is in a decreasing run, then $a_1+a_2+\ldots +a_{n-1}+(n-2)\equiv 0$ mod 2.
\end{corollary}

\begin{lemma}
\label{fact6}
Suppose that $\pi(\underline{a})=\pi(a_1\ldots a_n)$ is immediately
followed by $\pi(\underline{b})=\pi(b_1\ldots b_n)$ in the full Gray
code. Let $i$ be the unique index such that $b_i=a_i \pm 1$
and $b_j=a_j$ for all $j \neq i$.  Then if $i<n$, either $a_k=1$ for all
$k>i$ or $a_k=2k-1$ for all $k>i$. 
\end{lemma}
\begin{proof}  
By the recursive nature of the Gray code, for
each $k>i$ we must have either $a_k=1$ or $a_k=2k-1$ .  In particular,
$a_n=1$ or $2n-1$ if $i<n$.  We would like to show that in a single word
we cannot have both $a_k=1$ and $a_j=2j-1$ for some $k,j>i$. Assume by
way of contradiction, that there is such a change, then there is a least
$k>i$ such that exactly one of $a_k$ and $a_{k+1}$ is equal to one. By
the recursive nature of our Gray code we may assume $k=n-1$, since
$a_j=b_j$ for $j\geq k+2$ implies that $b_1\ldots b_{k+1}$ immediately
follows $a_1\dots a_{k+1}$ in the full Gray code of words.  
We will show by way of contradiction the impossibility of the case when
$a_n=2n-1$ and $a_{n-1}=1$. The case when $a_{n-1}=2n-3$ and $a_n=1$ is
completely analogous. 

Since $a_n=b_n=2n-1$, $\underline{a}$ is in an increasing run.  By
Corollary~\ref{run-parity}, $a_1+\ldots +a_{n-1}+(n-2) \equiv 1$ mod 2.
Thus, $a_1+\ldots +a_{n-2}+(n-3) \equiv 1$ mod 2, which means
$\tau(\underline{a})$ is in an increasing run in the Gray code for words
of length $n-1$.  Since $a_{n-1}=1$, we must have $b_{n-1}=2$, in
contradiction with $i<n-1$.
\end{proof}

\begin{corollary}
\label{parity}
If $\underline{a}=a_1\ldots a_k 1\ldots 1$ and the next word in the Gray
code for words is $a_1\ldots a'_k 1\ldots 1$, then $a'_k=a_k+1$ implies
that $a_k$ is even, and $a'_k=a_k-1$ implies that $a_k$ is odd. 
Similarly, if $\underline{a}=a_1\ldots a_k (2k+1)\ldots (2n-1)$ and the
next word in the Gray code for words is $a_1\ldots a'_k (2k+1)\ldots
(2n-1)$, then $a'_k=a_k+1$ implies that $a_k$ is odd, and $a'_k=a_k-1$
implies that $a_k$ is even. 
\end{corollary}
Lemma~\ref{fact-m} has the following consequence.
\begin{corollary}
\label{parity-truncated}
Under the conditions of Lemma~\ref{fact6}, if $\underline{a}=a_1\ldots
a_k 1\ldots 1$, then $a_1+\ldots +a_k+(k-1)=0$ mod 2;  
and if $\underline{a}=a_1\ldots a_k (2k+1)\ldots (2n-1)$, then
$a_1+\ldots +a_k+(k-1)=1$ mod 2. 
\end{corollary}

\begin{lemma}
\label{fact8:prelim}
In each run in the Gray code for words of length $n$, there exists a $k$
such that if $a_n \le 2k-2$, $\underline{a}$ is arc-connected; and if
$a_n \ge 2k-1$, $\underline{a}$ is arc-disconnected. This $k$ is the
least index $k'\leq n-1$ such that $a_j\geq 2k'-1$ holds for all $j\in \{k',
k'+1, \ldots,n-1\}$, if such an index exists; otherwise $k=n$. 
\end{lemma}
\begin{proof}
Consider a run in the Gray code for words of length $n$. All the words in
the run have the same truncated word $\tau(\underline{a})=a_1\ldots a_{n-1}$.  

{\noindent \it Case 1}: $\tau(\underline{a})$ is arc-connected.
Applying Corollary~\ref{cor:arc-conn} to $\tau(\underline{a})$ we see
that there is no $k'\leq n-1$ such that 
$a_j\geq 2k'-1$ holds for all $j\in \{k',k'+1, \ldots,n-1\}$. By the
same Corollary~\ref{cor:arc-conn}, $\underline{a}$ is arc-disconnected if
and only if there exists a $k$ such that $a_k=2k-1$ and $a_j \ge
a_k=2k-1$ for all $j>k$.  Since $a_1 \cdots a_{n-1}$ is arc-connected,
such a $k$ must satisfy  $k=n$.  Thus, $a_1 \cdots a_n$ is
arc-disconnected if and only if $a_k=a_n=2n-1$.\\ 

{\noindent\it Case 2}: $\tau(\underline{a})$ is arc-disconnected.
Applying Corollary~\ref{cor:arc-conn} to $\tau(\underline{a})$ yields
that there is a smallest $k'\leq n-1$
such that $a_{k'}=2k'-1$ and $a_j \ge 2k'-1$ for $k'<j<n$.  Clearly, if $a_n \ge
2k'-1$ then $\underline{a}$ is arc-disconnected. We are left to show
that $\underline{a}$ is arc-connected whenever $a_n \le 2k'-2$.
Assume, by way of contradiction, that $\underline{a}$ is
arc-disconnected for some $a_n \le 2k'-2$. By
Corollary~\ref{cor:arc-conn}, there is a $k''$ such that $a_j\geq
2k''-1$ holds for all $j\geq k''$. By the minimality of $k'$ we must
have $k''\geq k'$. On the other hand $a_n\geq 2k''-1$ and $a_n\leq
2k'-2$ imply $k''<k'$, a contradiction.    
\end{proof}

\begin{lemma}
\label{fact8}
Suppose $k$ is defined as in Lemma~\ref{fact8:prelim}. Then 
$\pi(a_1\ldots a_{k-1})$ is the first sign-connected component of the
standard permutation $\pi(a_1\ldots a_{n-1})$.   
\end{lemma}
\begin{proof}
Assume first there is a least index $k'\leq n-1$ such that $a_j\geq
2k'-1$ holds for all $j\in \{k', k'+1, \ldots,n-1\}$.
The we have $k=k'$ and, by Corollary~\ref{cor:arc-conn}, the
standard permutation $\pi(a_1\cdots a_{k-1})$ is sign-connected.  
Since $a_j\geq 2k-1$ holds for $j=k,\ldots,n-1$, the left endpoints of
the corresponding arcs are to the right of the arcs associated to
$\pi(a_1\cdots a_{k-1})$. Therefore $\pi(a_1\ldots a_{k-1})$ is the
first sign-connected component of $\pi(\tau(\underline{a}))$.  

Assume now that there is no $k'\leq n-1$ satisfying $a_j\geq
2k'-1$ for all $j\in \{k', k'+1, \ldots,n-1\}$. In this case $k=n$ and,
by Corollary~\ref{cor:arc-conn}, $\pi(\tau(\underline{a}))$ is
arc-connected. 
\end{proof}

\begin{lemma}
\label{fact5}
Suppose $\underline{a}$ is immediately followed by
$\underline{b}$ in the full Gray code for words and let $k$ be the
unique index such that $b_k\neq a_k$. 
If $b_k=a_k+1$ and $\underline{a}$ is arc-disconnected,  
then $\underline{b}$ is arc-disconnected; and 
if $b_k=a_k-1$ and $\underline{a}$ is arc-connected,  
then $\underline{b}$ is arc-connected.
\end{lemma}
\begin{proof}
Consider first the case when $k=n$. In this case $\underline{a}$ and
$\underline{b}$ belong to the same run and the statement follows from
Lemma~\ref{fact8:prelim}. We are left to consider the case when $k \neq
n$. In this case Lemma~\ref{fact6} tells us that either $a_i=1$ holds
for all $i > k$  or $a_i=2i-1$ holds for all $i > k$. By
Lemma~\ref{fact8:prelim}, $\underline{a}$ is arc-connected exactly when
$a_i=1$ holds for all $i>k$ and $\underline{a}$ is arc-disconnected exactly when
$a_i=2i-1$ holds for all $i>k$. The same characterization also applies
to $b$ since $a_i=b_i$ for all $i>k$.
\end{proof}

\section{Restricting the Gray Code to the Shelling Types of the $n$-Cube}
\label{sec:conn}

Our goal is to define a Gray code for the facet enumerations of the
boundary of the $n$-cube, restricted to shelling types.  This is
equivalent to finding a Gray code for the arc-connected words $a_1
\ldots  a_n$ since arc-connected words encode sign-connected standard
permutations, and we know the sign-connected standard permutations
represent shellings of the boundary of the $n$-cube.   

In this section, we will show that the sublist obtained by removing all
arc-disconnected words from the full Gray code of words from
Section~\ref{sec:gray} yields a Gray code for the sign-connected standard
permutations. This sublist of the standard permutations will be referred
to as the connected Gray code. Our main result is the following. 

\begin{theorem}
\label{prop1}
Suppose $\underline{a}$ and $\underline{b}$ are arc-connected, but every
code listed between these two words in the full Gray code of words is
arc-disconnected.  Then $\pi(\underline{a})$ and $\pi(\underline{b})$
differ by an adjacent transposition. 
\end{theorem}

\begin{proof}
Without loss of generality we may assume that $\underline{b}$ follows
$\underline{a}$ in the full Gray code of words. 
By Lemma~\ref{fact8:prelim}, $\underline{a}$ cannot be at the end of a
run and, Lemma~\ref{fact5}, $\underline{a}$ must be in an increasing
run. So $a_1\ldots a_{n-1}a_n$ is arc-connected, and the next consecutive
word in the Gray code of words is $a_1\ldots a_{n-1}a'_n$, where
$a'_n=a_n+1$.  This word is disconnected as are all the remaining codes
in the run up to and including $a_1\ldots a_{n-1}(2n-1)$.   

The next run in the Gray code of words starts with $c_1\ldots
c_{n-1}(2n-1)$ which is arc-disconnected.  This run decreases down to
$c_1\ldots c_{n-1}1$, an arc-connected code.  We must have $c_1\ldots
c_{n-1}=b_1\ldots b_{n-1}$ since $b_1\ldots b_n$ is the first
arc-connected code following $\underline{a}$.  Thus, the next run in the
Gray code of words actually starts with $b_1\ldots b_{n-1}(2n-1)$, and
the subsequent codes in the run down to $b_1\ldots b_{n-1}(b_n+1)$ are
arc-disconnected. 

By the construction of the Gray code, $\tau(\underline{a})$ and
$\tau(\underline{b})$ are consecutive words in the full Gray code for
words of length $n-1$; thus, there is exactly one $k<n$ such that
$b_k=a_k \pm 1$.  By Lemma~\ref{fact6}, either we have $k=n-1$ or we have
$k<n-1$ and $\tau(\underline{a})$ is either of the form 
$\tau(\underline{a})=a_1\ldots a_k 1 \ldots  1$, or of the form 
$\tau(\underline{a})=a_1\ldots a_k (2k+1) \ldots  (2n-3)$.

{\noindent\it Case 1:} $\tau(\underline{a})=a_1\ldots a_k 1 \ldots  1$
where $a_k\neq1$. By our assumption, $b_k=a_k \pm 1$.  Thus, both
$\tau(\underline{a})$ and $\tau(\underline{b})$ are arc-connected; by
Lemma~\ref{fact8}, we must have $a_n=2n-2=b_n$.  In the arc diagram,
$a_j=b_j=1$ for $k < j < n$ means the first $k$ arcs are under $n-k-1$
nested arcs.  An arc stretching over all previous arcs will not change
how any two ends of the previous arcs are interchanged since that move
occurs completely under the arcs.  The $n^{th}$ arc ($a_n=2n-2$) will
only overlap the $(n-1)^{st}$ arc, meaning this last arc will not affect
how the first $k$ arcs change when we make the move associated to
changing from $\underline{a}$ to $\underline{b}$. 
By the recursive definition of the full Gray code, we know that
$\pi(a_1\ldots a_k)$ and $\pi(b_1\ldots b_k)$ differ by an adjacent
transposition.  This means the arc diagrams encoded by $a_1\ldots a_k$
and $b_1\ldots b_k$ differ by exactly two adjacent ends of two distinct
arcs swapping positions; and since we already know that the $(k+1)^{st}$
through $n^{th}$ arcs will not affect that swap, we get that
$\underline{a}=a_1\ldots a_k 1 \ldots  1 (2n-2)$ and
$\underline{b}=b_1\ldots b_k 1 \ldots  1 (2n-2)$ encode two standard
permutations which differ by an adjacent transposition.  

{\noindent\it Case 2:} $\tau(\underline{a})=a_1\ldots a_k (2k+1) \ldots
(2n-3)$ where $a_k \neq 2k-1$. By Corollary~\ref{cor:arc-conn},
$\tau(\underline{a})$ is arc-disconnected, so there must exist a $j \le
k+1$ such that $a_j=2j-1$ and $a_i \ge 2j-1$ for $j<i<n$.  Let us choose
$j$ to be the smallest index with this property. By
Lemma~\ref{fact8:prelim}, the word
$a_1\ldots a_k(2k+1)\ldots (2n-3)(2j-2)$ is arc-connected, but the rest
of the codes in the run, $a_1\ldots a_k(2k+1)\ldots (2n-3)(2j-1)$
through $a_1\ldots a_k(2k+1)\ldots (2n-3)(2n-1)$, are arc-disconnected.
Thus, we must have $a_n=2j-2$.    

Since $\underline{b}$ is in a decreasing run 
immediately following the run in which $\underline{a}$ is contained, the
run of $\underline{b}$ starts with $b_1\ldots b_k(2k+1)\ldots
(2n-3)(2n-1)$ where $b_k=a_k \pm 1$.  By Corollary~\ref{parity}, if
$b_k=a_k+1$, then $a_k$ is odd, and if $b_k=a_k-1$, then $a_k$ is even.
Since $j \neq k$ and $j\leq k+1$, we have two subcases:   

{\noindent\it Case 2a:} $j=k+1$. In this case, $a_1\ldots a_k$ is
arc-connected and $a_n=2j-2=2k$. We claim that $b_1\ldots b_k$ is also
arc-connected. This is an immediate consequence of Lemma~\ref{fact5}
when $b_k=a_k-1$. If $b_k=a_k+1$, then $b_k$ is even and $b_1\ldots b_k$
is arc-connected because, by Lemma~\ref{fact8:prelim}, the first
arc-disconnected code in an increasing run ends with an odd letter. 
Thus, both $\underline{a}=a_1\ldots a_k (2k+1) \ldots  (2n-3) (2k)$ and
$b_1\ldots b_k (2k+1) \ldots  (2n-3) (2k)$ are
arc-connected codes, and there are only arc-disconnected codes between
these two words in the full Gray code of words. Therefore $b_n=2k$.

In the arc diagram of $\pi(\underline{a})=\pi(a_1\ldots a_k (2k+1)
\ldots  (2n-3)(2k))$, the arcs of $\pi(a_1\ldots a_k)$ form the first
connected component of $\pi(\tau(\underline{a}))$, which is followed by
$n-k-1$ minimal arcs (see Section~\ref{sec:arc} for a definition).  Then
the $n^{th}$ arc stretches over the minimal arcs to intersect only the
$k^{th}$ arc of the first connected component of
$\pi(\tau(\underline{a}))$.  Thus, the $(k+1)^{st}$ through $n^{th}$ 
arcs will not affect any moves among the first $k$ arcs. 
Since $\pi(a_1\ldots a_k)$ and $\pi(b_1\ldots b_k)$ must differ by an
adjacent transposition, the recursive construction of the full Gray code
guarantees that $\pi(\underline{a})$ and $\pi(\underline{b})$ will also
differ by an adjacent transposition. 

{\noindent\it Case 2b:} $j < k$. 
In this case, $a_1\ldots a_{j-1}$ is arc-connected and $2j-1 \le
a_k<2k-1$.  By Corollary~\ref{cor:arc-conn}, we have that $a_1\ldots
a_k$ is arc-disconnected and, by Lemma~\ref{fact8}, the first arc-connected
component of $\pi(a_1\ldots a_k)$ is $\pi(a_1\ldots a_{j-1})$. 
We claim that $b_1\ldots b_k$ is also arc-disconnected. This is an
immediate consequence of Lemma~\ref{fact5} when $b_k=a_k+1$.
If $b_k=a_k-1$ then $a_k$ is even and $b_1\ldots b_k$
is arc-disconnected because $a_k$ is strictly greater than
$2j-1$, implying $b_k \ge 2j-1$.  

Since $b_1\ldots b_k$ is arc-disconnected and the same holds for
$b_1\ldots b_m=a_1\ldots a_m$ for any $m$ strictly between $j$ and $k$,   
$\pi(b_1\ldots b_{j-1})$ is the first connected component of
$\pi(b_1\ldots b_k)$.  
Hence, both $\underline{a}=a_1\ldots a_j\ldots a_k (2k+1) \ldots  (2n-3)
(2j-2)$ and $b_1\ldots b_j\ldots b_k (2k+1) \ldots  (2n-3)
(2j-2)$ are arc-connected codes such that there are only
arc-disconnected words between them in the full Gray code of words. Thus
we have $b_n=2j-2$.

In the arc diagram of $\pi(\underline{a})=\pi(a_1\ldots a_k (2k+1)
\ldots  (2n-3)(2j-2))$, the arcs of $\pi(a_1\ldots a_{j-1})$ form the
first connected component of $\pi(\tau(\underline{a}))$; the 
$(k+1)^{st}$ through $(n-1)^{st}$ arcs are minimal arcs located at the
right end of the diagram; and the $n^{th}$ arc stretches over the second
through last connected components of the arc diagram, intersecting only
the $(j-1)^{st}$ arc of the first connected component.  Thus, the
$(k+1)^{st}$ through $n^{th}$ arcs will not affect any changes occurring
in the first $k$ arcs.  The recursive construction of the full
Gray code ensures that $\pi(\underline{a})$ and $\pi(\underline{b})$
will differ by an adjacent transposition since $\pi(a_1\ldots a_k)$ and
$\pi(b_1\ldots b_k)$ differ by an adjacent transposition.  

{\noindent\it Case 3:} $k=n-1$ (namely, $b_{n-1}=a_{n-1} \pm 1$). 
By Lemma~\ref{fact8:prelim}, there exists a $j$ such that $a_1\ldots a_{n-1}(2j-2)$ is arc-connected but that $a_1\ldots a_{n-1}(2j-1)$ is arc-disconnected.  
If $\tau(\underline{a})=a_1\ldots a_{n-1}$ and
$\tau(\underline{b})=b_1\ldots b_{n-1}$ are both arc-connected, this is
a degenerate case of case $2a$, where the number of minimal arcs to the
right of the first connected component of the arc diagram of
$\pi(\tau(\underline{a}))$ zero (since $\tau(\underline{a})$
is arc-connected). This does not change the conclusion of the argument.
Similarly, if $\tau(\underline{a})=a_1\ldots a_{n-1}$ and
$\tau(\underline{b})=b_1\ldots b_{n-1}$ are both arc-disconnected, we
have a degenerate case of case $2b$. We are left to consider the case
when $\tau(\underline{a})$ is 
arc-connected but $\tau(\underline{b})$ is arc-disconnected, and the
case when $\tau(\underline{a})$ is
arc-disconnected but $\tau(\underline{b})$ is arc-connected. 
We will show by way of contradiction that neither of these
cases can occur.

Assume first that $\tau(\underline{a})$ is
arc-connected, but is immediately 
followed by the arc-disconnected word $\tau(\underline{b})$ in the Gray
code for words of length $n-1$.  Then by Lemma~\ref{fact8:prelim},
$\tau(\underline{a})=a_1\ldots a_{n-2}(2j-2)$ for some $j$.  Since
$\underline{a}$ is in an increasing run, Corollary~\ref{run-parity}
gives $a_1+\ldots +a_{n-2}+(2j-2)+(n-2) \equiv 1$ mod 2.  Thus,
$a_1+\ldots +a_{n-2}+(n-3) \equiv 0$ mod 2, so $\tau(\underline{a})$ is
in a decreasing run in the Gray code of length $n-1$.  Hence,
$b_{n-1}=a_{n-1}-1=2j-3$.  By Lemma~\ref{fact5}, we get that
$\tau(\underline{b})$ is arc-connected, a contradiction.  

Assume finally that $\tau(\underline{a})$ is arc-disconnected, but is
immediately followed by the arc-connected word in $\tau(\underline{b})$
the Gray code for words of length $n-1$.  Then by
Lemma~\ref{fact8:prelim}, $\tau(\underline{a})=a_1\ldots a_{n-2}(2j-1)$
for some $j$.  Since $\underline{a}$ is in an increasing run,
Corollary~\ref{run-parity} gives $a_1+\ldots +a_{n-2}+(2j-1)+(n-2)
\equiv 1$ mod 2.  Thus, $a_1+\ldots +a_{n-2}+(n-3) \equiv 1$ mod 2, so
$\tau(\underline{a})$ is in an increasing run in the Gray code for words
of length $n-1$.  Hence, $b_{n-1}=a_{n-1}+1=2j$.  By Lemma~\ref{fact5}
we get that $\tau(\underline{b})$ is arc-disconnected, a contradiction. 
\end{proof}

In the proof of Theorem~\ref{prop1} we have also shown the following
statement. 
\begin{proposition}
\label{prop2}
Suppose $\underline{a}$ and $\underline{b}$ as in Theorem~\ref{prop1}.  Then
$a_n=b_n=2i-2$ for some $i\in \{2, \ldots, n\}$, satisfying $a_i=b_i=2i-1$ and
$a_j,b_j \ge 2i-1$ for all $j\in \{i,\ldots, n-1\}$. 
\end{proposition}

If we look at the list of words encoding the standard permutations of
the connected Gray code, then each run starts with $a_n=1$ and ends with
$a_n=2i-2$ for some $i$ (or vice versa).  The proof of
Theorem~\ref{prop1} allows us to describe the relationship between the
any arc-connected $\underline{a}$ and the arc-connected $\underline{b}$
immediately following it in the connected Gray code. We see that in the
case when $\underline{a}$ and $\underline{b}$ are in different runs of
the full Gray code, $\underline{a}$ and $\underline{b}$ have the
following properties:
\begin{enumerate}
\item $\pi(\underline{a})$ and $\pi(\underline{b})$ differ by an adjacent transposition,
\item $a_n=b_n$,
\item if $a_n=b_n=2i-1$ for some $i>1$ then $\tau(\underline{a})$ and $\tau(\underline{b})$ are either both connected or both disconnected,
\item if $a_n=b_n=2i-1$ for some $i>1$, $\tau(\underline{a})$ and
  $\tau(\underline{b})$ are both arc-disconnected and $\pi(a_1\ldots a_k)$
  is the first connected component of $\pi(\tau(\underline{a}))$, then
  the first connected component of $\pi(\tau(\underline{b}))$ is
  $\pi(b_1\ldots b_k)$.  
\end{enumerate}

As an immediate consequence of Theorem~\ref{prop1}, the sublist of the
full Gray code obtained by simply removing all of the sign-disconnected
standard permutations is also a Gray code.  Hence, by working with the
words which encode standard permutations, we have found a Gray code,
namely the connected Gray code, for the standard permutations which
represent the shelling types of the of the facets of the boundary of the
$n$-cube.


\end{document}